\documentclass[12pt]{amsart}
\usepackage{amssymb}
\usepackage{datetime}
\usepackage[english]{babel}
\usepackage{epsfig}
\usepackage{mathptmx}
\usepackage{amsmath,amsfonts,amssymb}
\usepackage{mathtools}
\usepackage{mathrsfs}
\usepackage[all]{xy}
\usepackage{graphicx}
\usepackage{latexsym}
\usepackage{verbatim}
\numberwithin{equation}{section}

\renewcommand {\Re} {{\rm Re}}
\renewcommand {\Im} {{\rm Im}}

\theoremstyle{theorem}

\newtheorem{theorem}{Theorem}[section]
\newtheorem{lemma}[theorem]{Lemma}

\newtheorem{corollary}[theorem]{Corollary}

\theoremstyle{definition}

\newtheorem{example}[theorem]{Example}

\theoremstyle{remark}
\newtheorem{remark}[theorem]{Remark}

\numberwithin{equation}{section}

\begin{document}
\title[Asymptotic dilation of regular homeomorphisms]{Asymptotic dilation of regular homeomorphisms}
\date{\today \hskip 3mm \currenttime \hskip 4mm (\texttt {AsymDil.tex})}

\author[A. Golberg]{Anatoly Golberg}

\author[R. Salimov]{Ruslan Salimov}

\author[M. Stefanchuk]{Maria Stefanchuk}

\begin{abstract} 
We study the asymptotic behavior of the ratio $|f(z)|/|z|$ as $z\to 0$ for mappings differentiable a.e. in the unit disc with non-degenerated Jacobian. The main tools involve the length-area functionals and angular dilatations depending on some real number $p.$  The results are applied to homeomorphic solutions of a nonlinear Beltrami equation. The estimates are illustrated by examples.
\end{abstract}
\maketitle

%%%%%%%%%%%%%%%%%%%%%%%%%%%%%%%%%%%%%%%%%%%%

\medskip

\bigskip

\section{Introduction}

\subsection{} A very interesting problem of classical Complex Analysis is to extend Schwarz's Lemma to more general classes of functions than analytic ones. A result of such type has been established by Teichm\"uller \cite{Tei38} for quasiconformal mappings. He proved the existence of limit of $|f(z)|/|z|$ at the origin and called it \emph{asymptotic dilation} assuming the finiteness of the integral of a form $\int_{|z|<1} (K_f(z)-1)/|z|^2 dxdy$. Here $K_f(z)$ stands for the real coefficient of quasiconformality of the mapping $f.$ Later the same result has been obtained by a different method (Ahlfors' inequalities) by Wittich \cite{Wit48}. For extensions of such results see also \cite{Bel74}, \cite{Leh60}, \cite{RW65} and others. The most of these results involve the module technique (extremal lengths).

In the paper we apply the length-area method originated in 1920's and isoperimetric inequality; see, e.g. \cite{AB50}, \cite {BGMR13}, \cite{Hay94}, \cite{Suv85}.

Throughout the paper we consider sense-preserving homeomorphisms from the Sobolev class $W^{1,1}$ in the unit disc requiring regularity and possessing Lusin's $(N)$-property (preservation of sets with zero measure). The regularity of a mapping in a domain means its differentiability almost everywhere (a.e.) and nonvanishing Jacobian a.e. Since we study the $W^{1,1}$-homeomorphisms, the differentiability a.e. follows \textit{a priori}; see, e.g. \cite{HK14}. Note that nonvanishing Jacobian is equivalent to Lusin's ($N^{-1}$)-property.

In this paper the asymptotic behavior (asymptotic dilation) of such mappings is studied under assumption that some integral mean of the angular dilatation is finite.

\subsection{} There are various counterparts of directional dilatations in the plane; see, e.g. \cite{AC71}, \cite{GMSV05}, \cite{RSY05}. These dilatations are much flexible than the classical (inner, outer, linear) ones and provide the important tools for investigation of local properties of quasiconformal mappings in the Eucledian and more general metric spaces (cf. \cite{Gol10}, \cite{GS14CSP}, \cite{GG09}).

We use the following definition. Let $D$ be a domain in $\mathbb C$ and $f\colon D \to \mathbb C$ be a regular homeomorphism of Sobolev class $W_{\rm{loc}}^{1, 1}.$ For any $p>1,$ the quantity
\begin{equation}\label{eq2}
D_p(z,0)=D_{p}(re^{i\theta},0)=\frac{|f_{\theta}(re^{i\theta})|^{p}}{r^pJ_f(re^{i\theta})}
\end{equation}
is called $p$-\emph{angular dilatation} of $f$ at a point $z\in D,$ $z\ne 0,$ with respect to the origin. For simplicity of notations, we write $D_p(z)=D_p(z,0).$

\subsection{} Let $B_r=\{z\in\mathbb C: |z|< r\},$ $\gamma_r=\{z\in\mathbb C: |z|=r\},$ and $\mathbb B=B_1.$
Our first result states:

\medskip
\begin{theorem}\label{th1}
Let $f\colon \mathbb B\to \mathbb B$ be a regular homeomorphism of Sobolev class $W^{1,1}_{\rm loc}$ possessing Lusin's $(N)$-property, and  $f(0)=0$. Suppose that there exists $k$ such that for $p>2,$
\begin{equation}\label{eqth1}
\liminf\limits_{r\rightarrow0}\left(\frac{1}{\pi r^2}\iint\limits_{B_{r}}D_{p}^{\frac{1}{p-1}}(z)\,dxdy\right)^{p-1}
\le k<\infty.
\end{equation}
Then
\begin{equation}\label{eqth2}
\liminf\limits_{z\rightarrow0}\frac{|f(z)|}{|z|}\le c_{p}\,k^{\frac{1}{p-2}}<\infty,
\end{equation}
where $c_{p}$ is a positive constant depending only on $p$.
\end{theorem}

Some results related to \eqref{eqth2} are given in \cite{GS14GMJ}, \cite{Ikoma65}, \cite{Sal15}.

\section{Auxiliary results}

We present several needed preliminary results. Some of those have an intrinsic interest.

\medskip
\subsection{} For a measurable function $Q\,:\,D\,\to\,[0,+\infty]$ and $p\,>\,1,$ denote
\begin{equation}\label{dqp}
q_{p}(r)=\left(\frac{1}{2\pi
r}\int\limits_{\gamma_{r}}Q^{\frac{1}{p-1}}(z)\,|dz|\right)^{p-1}\,.
\end{equation}
When $Q(z)=D_p(z),$ we write $q_p(r)=d_p(r).$

The following statement provides a differential inequality for the area functional $S(r)=m(f(B_r)),$ where $m$ denotes the two dimensional Lebesgue measure.

\medskip
\begin{lemma}\label{lem1}
Let $f\colon  \mathbb B\to \mathbb
B$ be a regular $W^{1, 1}_{\rm
loc}$-homeomorphism with Lusin's $(N)$-property, and $p > 1.$ Then for a.a. $r\in [0,1],$
\begin{equation}\label{eqlem1}
S'(r)\ge 2\pi^{\frac{2-p}{2}}\, r^{1-p}\, d_{p}^{-1}(r)\,
S^{\frac{p}{2}}(r).
\end{equation}
\end{lemma}

\begin{proof} Denote by $L(r)$ the length of curve $f(re^{i\theta}),$ $0\le\theta\le 2\pi.$ For a.a. $r\in [0,1],$
\begin{equation*}
L(r)=\int\limits_{0}^{2\pi}|f_{\theta}(re^{i\theta})|\,d\theta=\int_{0}^{2\pi}
\left[D_p(re^{i\theta})\right]^\frac{1}{p}
\left[J_f(re^{i\theta})\right]^\frac{1}{p}r\,d\theta\,,
\end{equation*}
and by H\"older's inequality,
\begin{equation}\label{eqlem11}
L^p(r)\,\le\,
\left(\int\limits_{0}^{2\pi}\left[D_p(re^{i\theta})\right]^\frac{1}{p-1}r\,d\theta\right)^{p-1}
\int\limits_{0}^{2\pi}J_f(re^{i\theta})r\,d\theta\,.
\end{equation}

Due to Lusin's $(N)$-property and the Fubini theorem,
\begin{equation*}
S(r)=\iint\limits_{B_r}J_{f}(z)\,dxdy=\int\limits_{0}^{r}\int\limits_{0}^{2\pi}J_{f}(te^{i\theta})t\,dtd\theta\,;
\end{equation*}
hence, for a.a. $r\in [0,1],$
\begin{equation*}
S^\prime(r)\,=\,\int\limits_{0}^{2\pi}
J_{f}(re^{i\theta}) r\,d\theta\,.
\end{equation*}
Estimating the last integral by (\ref{eqlem11}) and using (\ref{dqp}) for $Q(z)=D_p(z),$ one obtains
\begin{equation}\label{eqlem12}
S^\prime(r)\,\ge\,\frac{L^p(r)}{(2\pi r)^{p-1}\,d_p(r)}\,.
\end{equation}
Combining (\ref{eqlem12}) with the planar isoperimetric inequality $L^2(r)\ge 4\pi S(r)$ implies the desired relation (\ref{eqlem1}).
\end{proof}

The inequality (\ref{eqlem12}) allows us to derive a counterpart of the classical length-area principle for conformal and quasiconformal mappings; see, e.g. \cite{Suv85}.

Integrating (\ref{eqlem12}), one gets for $0<r_1<r<r_2<1,$
\begin{equation*}
\int\limits_{r_1}^{r_2} \frac{L^p(r)\,dr}{(2\pi r)^{p-1}\,d_p(r)}\,\le\,S(r_2)-S(r_1)\,\le\,S(r_2)\,.
\end{equation*}

\subsection{} The upper bounds for the images of $B_r$ under regular homeomorphisms of Sobolev class $W^{1,1}$ are established in the following statement for any $p>2.$ For the case $p=2,$ we refer to \cite[Proposition~3.7]{BGMR13}.

\medskip
\begin{lemma}\label{lem2}
Let $f\colon  \mathbb B\to \mathbb
B$ be a regular $W^{1, 1}_{\rm
loc}$-homeomorphism with Lusin's $(N)$-property, and $p > 2.$ Then for a.a. $r\in [0,1],$
\begin{equation}\label{eqlem21}
S(r)\,\le\,\pi\,(p-2)^{-\frac{2}{p-2}}\left(\int\limits_r^1\frac{dt}{t^{p-1}\, d_{p}(t)}\right)^{-\frac{2}{p-2}}\,.
\end{equation}
\end{lemma}

\begin{proof}
By (\ref{eqlem1}) for a.a. $r\in [0,1],$
\begin{equation*}
\frac{S'(r)\,dr}{S^{\frac{p}{2}}(r)}\,\ge\,
2\pi^{\frac{2-p}{2}}\frac{dr}{r^{p-1}\,d_{p}(r)}\,,
\end{equation*}
and integrating over the segment $[r,1],$ we obtain
\begin{equation*}
\int\limits_{r}^{1}g_p^\prime(t)\,dt\,\ge\,
2\pi^{\frac{2-p}{2}}\int\limits_{r}^{1}\frac{dt}{t^{p-1}d_{p}(t)}\,,
\end{equation*}
where $g_p(t)=\frac{2}{2-p}S^\frac{2-p}{2}(t).$ This function is nondecreasing on $[0,1],$ and
\begin{equation*}
\int\limits_{r}^{1}g_{p}^\prime(t)dt\,\le\,
g_{p}(1)-g_{p}(r)=\frac{2}{2-p}\left(S^{\frac{2-p}{2}}(1)-S^{\frac{2-p}{2}}(r)\right)\,.
\end{equation*}
Combining the last two inequalities, we have
\begin{equation*}%\label{eqlem23}
\frac{S^{\frac{2-p}{2}}(1)-S^{\frac{2-p}{2}}(r)}{2-p}\,\ge
\pi^{\frac{2-p}{2}}\,
\int\limits_{r}^{1}\frac{dt}{t^{p-1}\, d_{p}(t)}\,.
\end{equation*}
This implies the inequality (\ref{eqlem21}), because for $p>2,$
\begin{equation*}
S^{\frac{2-p}{2}}(r)\,\ge\,
S^\frac{2-p}{2}(r)-S^\frac{2-p}{2}
(1)\,\ge\, \pi^\frac{2-p}{2}\,
(p-2)\int\limits_{r}^{1}\frac{dt}{r^{p-1}\,
d_{p}(t)}\,.
\end{equation*}
\end{proof}

\subsection{} The following result estimates the integrals of the type (\ref{dqp}) in the terms of integral averages for $Q.$

\medskip
\begin{lemma}\label{lem3}
Let $p>1$ and $Q\colon  \mathbb B\to [0,\infty]$ be locally integrable with the exponent $\frac{1}{p-1}$ in $\mathbb B.$ Then for any $\varepsilon\in[0;\frac{1}{2}]$ the following estimate
\begin{equation*}\label{eq1lem4}
\left(\int\limits_{\varepsilon}^{2\varepsilon}\frac{dr}{r^{p-1}q_{p}(r)}\right)^{-1}\le 2^{p-1}\,
\varepsilon^{p-2}\,
\left(\frac{1}{4\pi\varepsilon^{2}}\iint\limits_{B_{2\varepsilon}}Q^{\frac{1}{p-1}}(z)\,dxdy\right)^{p-1},
\end{equation*}
holds, where $q_{p}(r)$ is defined by (\ref{dqp}).
\end{lemma}

\begin{proof}
Notify that
\begin{equation*}
\varepsilon=\int\limits_{\varepsilon}^{2\varepsilon}\frac{1}{\left(\int\limits_{\gamma_{r}}Q^{\frac{1}{p-1}}(z)\,|dz|\right)^{\frac{p-1}{p}}}
\cdot\left(\int\limits_{\gamma_{r}}Q^{\frac{1}{p-1}}(z)\,|dz|\right)^{\frac{p-1}{p}}dr\,,
\end{equation*}
one obtains by H\"older's inequality with the  exponents $p$ and $p^{'}=p/(p-1),$
\begin{equation*}
\varepsilon^{p}\le\int\limits_{\varepsilon}^{2\varepsilon}\frac{dr}{\left(\int\limits_{\gamma_{r}}Q^{\frac{1}{p-1}}(z)\,|dz|\right)^{p-1}}
\cdot\left(\int\limits_{\varepsilon}^{2\varepsilon}\int\limits_{\gamma_{r}}Q^{\frac{1}{p-1}}(z)\,|dz|\,dr\right)^{p-1}.
\end{equation*}
Now by the Fubini theorem,
\begin{equation*}
\left(\int\limits_{\varepsilon}^{2\varepsilon}\frac{dr}
{\left(\int\limits_{\gamma_{r}}Q^{\frac{1}{p-1}}(z)\,|dz|\right)^{p-1}}\right)^{-1}\le 4^{p-1}\pi^{p-1}\varepsilon^{p-2}\left(\frac{1}{\pi(2\varepsilon)^{2}}\iint\limits_{B_{2\varepsilon}}Q^
{\frac{1}{p-1}}(z)\,dx\,dy\right)^{p-1},
\end{equation*}
which completes the proof.
\end{proof}

\subsection{} We also provide a lower bound for $S(r)$ when $1<p<2.$

\medskip
\begin{lemma}\label{lem4}
Let $f\colon  \mathbb B\to \mathbb
B$ be a regular $W^{1, 1}_{\rm
loc}$-homeomorphism possessing Lusin's $(N)$-property, and $1<p<2.$ Then for a.a. $r\in [0,1],$
\begin{equation}\label{eqlem4}
S(r)\,\ge\,\pi\,(2-p)^\frac{2}{2-p}\left(\int\limits_0^r\frac{dt}{t^{p-1}\, d_{p}(t)}\right)^\frac{2}{2-p}\,.
\end{equation}
\end{lemma}

\begin{proof}
Integrating (\ref{eqlem1}) along the segment $[0,r],$ and arguing similar to the proof of Lemma~\ref{lem2}, we arrive at
\begin{equation*}
\frac{S^{\frac{2-p}{2}}(r)-S^{\frac{2-p}{2}}(0)}{2-p}\,\ge
\pi^{\frac{2-p}{2}}\,
\int\limits_{0}^{r}\frac{dt}{t^{p-1}\, d_{p}(t)}\,,
\end{equation*}
which implies (\ref{eqlem4}), because $S(0)=0.$
\end{proof}

\medskip
Lemma~\ref{lem4} yields that for any homeomorphism $f\colon \mathbb B\to \mathbb B$ satisfying its assumptions, the integral $\int_0^1\,dt/(t^{p-1}d_p(t))$ always converges.

\medskip
Indeed, for $1<p<2,$ we have the estimate
\begin{equation*}
\pi(2-p)^\frac{2}{2-p}\,\left(\int\limits_0^r\frac{dt}{t^{p-1}d_p(t)}\right)^\frac{2}{2-p}\,\le\,S(r)\,\le\,\pi\,,
\end{equation*}
which implies
\begin{equation*}
\int\limits_0^r\frac{dt}{t^{p-1}d_p(t)}\,\le\,\frac{1}{2-p}\,,\quad \forall r\in (0,1).
\end{equation*}
Now letting $r\uparrow 1,$ one obtains
\begin{equation*}
\int\limits_0^1\frac{dt}{t^{p-1}d_p(t)}\,\le\,\frac{1}{2-p}\,<\,\infty.
\end{equation*}

\section{Proofs of upper bounds and illustrating examples}

\subsection{} We start with the proof of Theorem~\ref{th1}.

\begin{proof} Let $l_{f}(r)=\min\limits_{|z|=r}|f(z)|$. Since $f(0)=0,$ $\pi l^{2}_{f}(r)\le S(r)$, or equivalently,
\begin{equation}\label{eqth3}
l_{f}(r)\le\sqrt{\frac{S(r)}{\pi}}.
\end{equation}

It follows from combining Lemmas~\ref{lem2} and \ref{lem3} that for $p>2,$
\begin{equation}\label{eqth4}
\frac{S(r)}{\pi r^{2}}\le c^\prime_{p}\,\left(\frac{1}{4\pi r^2}
\iint\limits_{B_{2r}} D_{p}^{\frac{1}{p-1}}(z)\,dxdy\right)^{\frac{2(p-1)}{p-2}},
\end{equation}
where $r\in[0,\frac{1}{2}]$ and $c^\prime_{p}$ is a positive constant depending only on $p.$ Here we have weakened the inequality (\ref{eqlem21}) replacing the limits $r$ and $1$ by $r$ and $2r,$ $r\in [0,\frac{1}{2}].$

Now by (\ref{eqth3}), (\ref{eqth4}) and (\ref{eqth1}),
\begin{equation*}
\begin{split}
\liminf\limits_{z\rightarrow0}\frac{|f(z)|}{|z|}&=\liminf\limits_{r\rightarrow0}\frac{l_{f}(r)}{r}
\le\liminf\limits_{r\rightarrow0}\sqrt{\frac{S(r)}{\pi r^{2}}}
\\&\le\liminf\limits_{r\rightarrow0}\sqrt{c_{p}^{'}\left(\frac{1}{4\pi r^2}
\iint\limits_{B_{2r}}D_{p}^{\frac{1}{p-1}}(z)\,dxdy\right)^{\frac{2(p-1)}{p-2}}}\leq
c_{p}\,k^{\frac{1}{p-2}},
\end{split}
\end{equation*}
with $c_{p}>0$ depending only on $p.$ This completes the proof.
\end{proof}

\subsection{} We now illustrate the sharpness of the growth order in the inequality (\ref{eqth2}).

\medskip
\begin{example}\label{example1}
Consider $f\colon \mathbb{B} \to
\mathbb{B}$ given by $f(z)=kz,$ $0<|k|\le 1.$
\end{example}

By a direct calculation, letting $z=re^{i\theta},$ the partial derivatives
\begin{equation*}
f_{\theta}=kire^{i\theta}, \quad f_{r}=ke^{i\theta}
\end{equation*}
and
\begin{equation*}
J_{f}\left(re^{i\theta}\right)=\frac{1}{r}\Im\left(f_{\theta}\cdot \overline{f_{r}}\right) =|k|^{2}\,.
\end{equation*}

Hence,
\begin{equation*}
D_{p}\left(re^{i\theta}\right)=\frac{|f_{\theta}|^{p}}
{r^{p}J_{f}\left(re^{i\theta}\right)}=|k|^{p-2}\,,
\end{equation*}
and therefore,
\begin{equation*}
\lim\limits_{r\to 0}\left(\frac{1}{\pi r^2}\iint\limits_{B_r}D_{p}^{\frac{1}{p-1}}(z)\,dxdy\right)^{p-1}=|k|^{p-2}.
\end{equation*}

On the other hand,
\begin{equation*}
\lim\limits_{z\to0}\frac{|f(z)|}{|z|}=|k|.
\end{equation*}

\subsection{} The next example shows that the condition (\ref{eqth1}) in Theorem~\ref{th1} cannot be omitted.

\medskip
\begin{example}
Assume that $p>2$ and  $f$ is an automorphism of the unit disc $\mathbb{B}$ of the form
\begin{equation*}
f(z)=\left(1+(p-2)\,\int\limits_{|z|}^{1}\frac{dt}{t^{p-1}\left(\ln\frac{e}{t}\right)^{p-1}}\right)^{\frac{1}{2-p}}\,\frac{z}{|z|},\quad\text{for}\quad z\ne 0,
\end{equation*}
and $f(0)=0.$
\end{example}

We shall use the notations
\begin{equation*}
I(r)=1+(p-2)\,\int\limits_{r}^{1}\frac{dt}{t^{p-1}\left(\ln\frac{e}{t}\right)^{p-1}}
\end{equation*}
for $r\ne 0,$ $\,R(r)=I^{-\frac{1}{p-2}}(r),$ if $r\ne 0,$ and extend $R(r)$ to the origin by $R(0)=0.$

Then $f(z)=R(r)e^{i\theta},$ where $z=re^{i\theta}.$ Note that $f$ is a diffeomorphism in $\mathbb{B}\setminus\{0\},$ which implies that $f\in W_{\rm loc}^{1,2}(\mathbb{B}\setminus\{0\}),$ although $f$ belongs to  $W^{1,1}$ in the whole disc $\mathbb B.$ Denote by $\lambda_1$ and $\lambda_2$ the maximal and minimal stretchings of $f$ at $z=re^{i\theta},$ $0<r<1,$ then $\lambda_1=R(r)/r,$ $\lambda_2=R^\prime(r)=R(r)r^{1-p}\ln^{1-p}(e/r)I^{-1}(r),$ and
\begin{equation*}
\iint\limits_{B_{r}}\lambda_1\,dxdy=\iint\limits_{B_{r}}\frac{R(|z|)}{|z|}\,\,dxdy\le
M_{r}\iint\limits_{B_{r}}\frac{dxdy}{|z|}=2\pi r M_{r}<\infty,
\end{equation*}
where $M_{r}=\max\limits_{z\in \overline B_{r}}|f(z)|$, $\overline B_{r}=\{z\in\mathbb{C}\,:\,|z|\le r\}.$
By a direct calculation, $f_{r}=R^{'}(r)e^{i\theta},$ $f_{\theta}=iR(r)e^{i\theta},$ and
\begin{equation*}
J_{f}(re^{i\theta})=\frac{1}{r}\Im\left(f_{\theta}\cdot
\overline{f_{r}}\right)=\frac{R(r) R^\prime(r)}{r}.
\end{equation*}
Hence,
\begin{equation*}
D_{p}(re^{i\theta})=\frac{|f_{\theta}|^{p}}{r^{p}J_{f}(re^{i\theta})}=\frac{R^{p-1}(r)}{r^{p-1}R^{'}(r)}=\left(\ln\frac{e}{r}\right)^{p-1}\,,
\end{equation*}
and therefore,
\begin{equation*}
\liminf\limits_{r\rightarrow 0}\frac{1}{\pi r^2}\iint\limits_{B_{r}}
D^{\frac{1}{p-1}}_{p}(z)\,dxdy=\infty.
\end{equation*}
Now instead of (\ref{eqth2}), we have by L'Hopital's rule,
\begin{equation*}
\liminf\limits_{z\rightarrow0}\frac{|f(z)|}{|z|}=\infty.
\end{equation*}

\subsection{} As the consequences of Theorem~\ref{th1}, we obtain

\medskip
\begin{corollary}
If $D_{p}(z)\le k<\infty$ for a.a. $z\in \mathbb{B}$, then
\begin{equation*}
\liminf\limits_{z\rightarrow0}\frac{|f(z)|}{|z|}\leqslant c_{p}\,k^{\frac{1}{p-2}},
\end{equation*}
with a positive constant $c_{p}$ depending only on $p.$
\end{corollary}

\medskip
\begin{corollary}
Let $f\colon \mathbb B\to \mathbb B$ be a regular homeomorphism of Sobolev class $W^{1, 1}_{\rm loc},$
possessing Lusin's $(N)$-property and satisfying $f(0)=0.$ Assume that $p>2$
and
\begin{equation*}%\label{eq1lem514}
\liminf\limits_{r\rightarrow0}\frac{1}{\pi r^2}\iint\limits_{B_{r}}D_{p}^{\frac{1}{p-1}}(z)\,dxdy=0.
\end{equation*}
Then
\begin{equation*}%\label{eq1lem515}
\liminf\limits_{z\rightarrow0}\frac{|f(z)|}{|z|}=0.
\end{equation*}
\end{corollary}

We illustrate the last corollary by

\medskip
\begin{example}
Take a radial stretching $f(z)=z|z|^{\alpha}$, $\alpha>0,$ of the unit disc $\mathbb B.$
\end{example}

\medskip
Then, letting $z=re^{i\theta}$, $f(z)=r^{\alpha+1}e^{i\theta},$
\begin{equation*}
f_{\theta}=ir^{\alpha+1}e^{i\theta},\quad
f_{r}=(\alpha+1)r^{\alpha}e^{i\theta}
\end{equation*}
and
\begin{equation*}
J_{f}\left(re^{i\theta}\right)=\frac{1}{r}\Im\left(f_{\theta}
\cdot \bar f_r\right)=(\alpha+1)r^{2\alpha}.
\end{equation*}
Therefore,
\begin{equation*}
D_{p}\left(re^{i\theta}\right)=\frac{|f_{\theta}|^{p}}{r^{p}J_{f}\left(re^{i\theta}\right)}=\frac{1}{\alpha+1}\cdot
r^{\alpha(p-2)},
\end{equation*}
and for $p>2,$
\begin{equation*}
\lim\limits_{r\to 0}\frac{1}{\pi r^2}\iint\limits_{B_r}D_{p}^{\frac{1}{p-1}}(z)\,dxdy=0.
\end{equation*}
On the other hand,
\begin{equation*}
\lim\limits_{z\to0}\frac{|f(z)|}{|z|}=0.
\end{equation*}

\subsection{} The asymptotic dilation of regular homeomorphisms $f\colon \mathbb B\to \mathbb B$ can be established also by applying a different assumption, which we call $\limsup$ condition.

\medskip
\begin{theorem}\label{th3}
Let $f\colon \mathbb B\to \mathbb B$ be a regular homeomorphism of Sobolev class $W^{1,1}_{\rm loc}$ possessing Lusin's $(N)$-property and normalized by $f(0)=0$. Suppose that $p>2$ and there exists $k_0$ such that
\begin{equation*}
k_0=\limsup\limits_{r\rightarrow0}\,r^{p-2}\int\limits_r^1 \frac{dt}{t^{p-1}d_p(t)}.
\end{equation*}
Then
\begin{equation}\label{eq1th3}
\liminf\limits_{z\rightarrow0}\frac{|f(z)|}{|z|}\le (p-2)^\frac{1}{2-p}\,k_0^{\frac{1}{2-p}}.
\end{equation}
\end{theorem}

\begin{proof}
Arguing similarly to the beginning of the proof of Theorem~\ref{th1} and applying the upper bound (\ref{eqlem21}) for $S(r),$ one gets
\begin{equation*}
\begin{split}
\liminf\limits_{z\rightarrow0}\frac{|f(z)|}{|z|}&\le\liminf\limits_{r\rightarrow0}\sqrt{\frac{S(r)}{\pi r^{2}}}
\\&\le\liminf\limits_{r\rightarrow 0}\frac{(p-2)^\frac{1}{2-p}}{r}\,\left(
\int\limits_r^1 \frac{dt}{t^{p-1}d_p(t)}\right)^\frac{1}{2-p}\,=\,
(p-2)^\frac{1}{2-p}\,k_0^\frac{1}{2-p}.
\end{split}
\end{equation*}
\end{proof}

The sharpness of the upper bound (\ref{eq1th3}) in Theorem~\ref{th3} is illustrated by Example~\ref{example1}. Indeed, by (\ref{dqp}) $D_p(z)=|k|^{p-2},$ thus, $d_p(t)=|k|^{p-2}$ for all $t\in(0,1).$ Therefore,
\begin{equation*}
\limsup\limits_{r\rightarrow0}\,r^{p-2}\int\limits_r^1 \frac{dt}{t^{p-1}d_p(t)}=\frac{1}{(p-2)|k|^{p-2}}.
\end{equation*}
Denoting $k_0=\frac{1}{(p-2)|k|^{p-2}},$ we have $|k|=(p-2)^\frac{1}{2-p}\,k_0^\frac{1}{2-p},$ which coincides with $\lim_{z\to 0}|f(z)|/|z|.$

\section{Lower and upper bounds for the ratio $|f(z)|/|z|$ and for Jacobian}

\subsection{} The asymptotic dilation of regular homeomorphisms $f\colon \mathbb B\to \mathbb B$ can be derived by applying the following $\limsup$ conditions.

\medskip
\begin{theorem}\label{th5}
Let $f\colon \mathbb B\to \mathbb B$ be a regular homeomorphism of Sobolev class $W^{1,1}_{\rm loc}$ possessing Lusin's $(N)$-property, and $f(0)=0$. Suppose that $1<p<2$ and there exists $k_0$ such that
\begin{equation*}
k_0=\limsup\limits_{r\rightarrow0}\,r^{p-2}\int\limits_0^r \frac{dt}{t^{p-1}d_p(t)}\,.
\end{equation*}
Then
\begin{equation}\label{eq1th5}
\limsup\limits_{z\rightarrow0}\frac{|f(z)|}{|z|}\ge (2-p)^\frac{1}{2-p}\,k_0^{\frac{1}{2-p}}.
\end{equation}
\end{theorem}

\begin{proof}
Denote $\mathcal L_f(r)=\max\limits_{|z|=r}|f(z)|$. Since $f(0)=0,$ we have $\pi \mathcal L^{2}_{f}(r)\ge S(r)$, or equivalently,
\begin{equation*}%\label{eqth5}
\mathcal L_{f}(r)\ge\sqrt{\frac{S(r)}{\pi}}\,.
\end{equation*}
Thus,
\begin{equation*}
\begin{split}
\limsup\limits_{z\rightarrow0}\frac{|f(z)|}{|z|}&\ge\limsup\limits_{r\rightarrow0}\sqrt{\frac{S(r)}{\pi r^{2}}}
\\&\ge\limsup\limits_{r\rightarrow 0}\frac{(2-p)^\frac{1}{2-p}}{r}\,\left(
\int\limits_0^r \frac{dt}{t^{p-1}d_p(t)}\right)^\frac{1}{2-p}\,=\,
(2-p)^\frac{1}{2-p}\,k_0^\frac{1}{2-p}.
\end{split}
\end{equation*}
\end{proof}

We illustrate the sharpness of the lower bound (\ref{eq1th5}) in Theorem~\ref{th5} by the same example (Example~\ref{example1}). Indeed, the condition (\ref{eq1th5}) is fulfilled with $k_0=|k|^{2-p}(2-p)^{-1},$ and $\lim_{r\to 0}|f(z)|/|z|=|k|<\infty.$

\subsection{} As a consequence of Theorems~\ref{th5} and \ref{th3}, one obtains, when the ratio $|f(z)|/|z|$ has only one limit point, the following statement.

\medskip
\begin{theorem}\label{th6}
Let $f\colon \mathbb B\to \mathbb B$ be a regular homeomorphism of Sobolev class $W^{1,1}_{\rm loc}$ possessing Lusin's $(N)$-property, and $f(0)=0$. Suppose that $1<p<2$ and there exist $k_1$ and $k_2$ such that
\begin{equation*}
k_1=\limsup\limits_{r\rightarrow0}\,r^{p-2}\int\limits_0^r \frac{dt}{t^{p-1}d_p(t)}\,,
\end{equation*}
\begin{equation*}
k_2=\limsup\limits_{r\rightarrow0}\,r^{p^\prime-2}\int\limits_r^1 \frac{dt}{t^{p^\prime-1}d_{p^\prime}(t)}\,.
\end{equation*}
Then, if $\lim\limits_{z\rightarrow0}|f(z)|/|z|=A$, the following double estimate
\begin{equation}\label{eq1th6}
(2-p)^\frac{1}{2-p}\,k_1^\frac{1}{2-p}\,\le\,A\,\le\, (p^\prime-2)^\frac{1}{2-p^\prime}\,k_2^\frac{1}{2-p^\prime}\,,
\end{equation}
holds. Here $p^\prime=p/(p-1),$  $2<p^\prime<\infty.$
\end{theorem}

\begin{proof}
The proof immediately follows from the assertions of the above theorems for $p,$ $1<p<2,$ and $p^\prime=p/(p-1),$ $2<p^\prime<\infty.$
\end{proof}

\begin{remark} The inequality (\ref{eq1th6}) provides the following relation between the constants $k_1$ and $k_2$ in Theorem~\ref{th6},
\begin{equation*}
k_1\le \frac{(p-1)^{p-1}}{(2-p)^p\,k_2^{p-1}}, \quad 1<p<2\,.
\end{equation*}
\end{remark}

\subsection{} The estimates established in Lemmas~\ref{lem2} and \ref{lem4} imply the following result concerning the area derivative and the Jacobian at the origin.

\medskip
\begin{theorem}\label{th7}
Let $f\colon \mathbb B\to \mathbb B$ be a regular homeomorphism of Sobolev class $W^{1,1}_{\rm loc}$ possessing Lusin's $(N)$-property. Suppose that for real numbers $p$ and $s,$ $1<p<2<s<\infty,$ there exist and coincide the following limits
\begin{equation*}
\lim\limits_{r\rightarrow0}\,(2-p)^\frac{2}{2-p}\left(r^{p-2}\int\limits_0^r\frac{dt}{t^{p-1}\, d_{p}(t)}\right)^\frac{2}{2-p}\,=\,
\lim\limits_{r\rightarrow0}\,(s-2)^\frac{2}{2-s}\left(r^{s-2}\int\limits_r^1\frac{dt}{t^{s-1}\, d_{s}(t)}\right)^\frac{2}{2-s}\,=\,A.
\end{equation*}
Then, the area derivative $\mu_f$ exists at 0, and
\begin{equation*}
\mu_f(0)=\lim\limits_{r\rightarrow0}\frac{S(r)}{\pi r^2}=A\,.
\end{equation*}
\end{theorem}

\begin{proof}
Applying Lemmas~\ref{lem2} and \ref{lem4} to $S(r)/\pi r^2,$ and letting $r\to 0,$ one obtains the assertion of Theorem~\ref{th7}.
\end{proof}

\begin{corollary}
If $z_0=0$ is a point of differentiability of $f,$ then $J_f(0)=A.$
\end{corollary}

\section{Nonlinear Beltrami equation}

\subsection{}
Let $D$ be a domain in $\mathbb{C}$ and $\mu\colon D\to \mathbb{C}$ be a measurable function with $|\mu(z)|<1$ a.e. in $D.$
The linear PDE
\begin{equation}\label{eq1lem402}
f_{\overline{z}}=\mu(z)f_{z}.
\end{equation}
is called the Beltrami equation; here $z=x+iy,$
\begin{equation*}
f_{\overline{z}}=\frac{1}{2}(f_{x}+if_{y}),\quad f_{z}=\frac{1}{2}(f_{x}-if_{y}).
\end{equation*}
The function $\mu$ is called the complex dilatation, and
\begin{equation*}
K_{\mu}(z)=\frac{1+|\mu(z)|}{1-|\mu(z)|}
\end{equation*}
is the real dilatation coefficient (or Lavrentiev coefficient) of \eqref{eq1lem402}.

For the case when $K_{\mu}$ is not essentially bounded, i.e. $||K_{\mu}||_{\infty}=\infty,$ \eqref{eq1lem402} is called the degenerate Beltrami equation.

\subsection{} Let $\sigma\colon D \to\mathbb{C}$ be a measurable function, and $m\ge 0.$ We consider the following nonlinear equation
\begin{equation}\label{nbp}
f_{r}=\sigma(re^{i\theta})\,|f_{\theta}|^{m}\, f_{\theta},
\end{equation}
written in the polar coordinates $(r,\theta).$ Here $f_{r}$ and $f_{\theta}$ are the partial derivatives of $f$ in $r$ and $\theta,$ respectively,
satisfying
\begin{equation*}
r f_{r}=zf_{z}+\bar{z}f_{\bar{z}} \, ,\, f_{\theta}=i(zf_{z}-\bar{z}f_{\bar{z}}).
\end{equation*}

The equation \eqref{nbp} in the Cartesian coordinates has the form
\begin{equation}\label{nb1}
f_{\bar{z}}=\,\frac{A(z)|zf_{z}-\bar{z}f_{\bar{z}}|^{m}-1}{ A(z)|zf_{z}-\bar{z}f_{\bar{z}}|^{m}+1}\, \frac{z}{\bar{z}} \, f_{z}\,,
\end{equation}
where $A(z)=\sigma(z)\,|z|\,i,$ or
\begin{equation*}%\label{nb2}
f_{\bar{z}}=\,\frac{A(z)\left(|z|^2|f_{z}|^2 -2\Re (z^2\,f_{z}\bar{f}_{\bar{z}} ) +|z|^2|f_{\bar{z}}|^2\right)^{m/2}-1}{ A(z)\left(|z|^2|f_{z}|^2 -2\Re (z^2\,f_{z}\bar{f}_{\bar{z}} ) +|z|^2|f_{\bar{z}}|^2\right)^{m/2}+1}\, \frac{z}{\bar{z}} \, f_{z}\,.
\end{equation*}

Note that in the case $m=0,$ the equation \eqref{nbp} is the usial Beltrami equation \eqref{eq1lem402} with the complex dilatation
\begin{equation*}
\mu(z)=\frac{z}{\bar{z}}\,\frac{A(z)-1}{A(z)+1}\,.
\end{equation*}

\subsection{} Now we present an application of the main result to homeomorphic solutions of the nonlinear equation \eqref{nbp}.

\medskip
\begin{theorem}
Let $f\colon \mathbb{B}\to \mathbb{B}$ be a regular homeomorphic solution of the equation \eqref{nbp} which belongs to Sobolev class $W^{1,2}_{\rm loc},$ and normalized by $f(0)=0.$ Assume that the coefficient $\sigma:\mathbb{B}\to \mathbb{C} $ satisfies the following condition
\begin{equation}\label{condnb1}
\liminf\limits_{r\rightarrow 0}\left(\frac{1}{\pi r^2}\iint\limits_{B_{r}} \frac{dxdy}{|z| \left( \Im\, \overline{\sigma}(z)\right)^{\frac{1}{m+1}}}\, \right)^{m+1}\le \sigma_{0}<\infty\,.
\end{equation}
Then
\begin{equation*}
\liminf\limits_{z\rightarrow 0}\frac{|f(z)|}{|z|}\le c_{m}\,\sigma_{0}^{\frac{1}{m}}<\infty\,,
\end{equation*}
where $c_{m}$ is a positive constant depending on the parameter $m.$
\end{theorem}

\begin{proof}
In the polar coordinates, the Jacobian of a sense-preserving mapping $f$ has the form $J_f(re^{i\theta})=\frac{1}{r}\Im(\bar f_r\cdot f_\theta),$ then the equalities \eqref{eq2} and \eqref{nbp} yield,
\begin{equation*}
D_p(z)=\frac{|f_\theta|^p}{r^{p-1}|f_\theta|^{m+2}\Im(\bar\sigma(z))}.
\end{equation*}
Letting $p=m+2,$ one rewrites the condition (\ref{eqth1}) as
\begin{equation*}
\liminf\limits_{r\rightarrow 0}\left(\frac{1}{\pi r^2}\iint\limits_{B_r}\frac{dxdy}{|z|\left(\Im(\bar\sigma(z))\right)^\frac{1}{m+1}}\right)^{m+1}\,\le\,\sigma_0\,.
\end{equation*}
Then, by Theorem~\ref{th1},
\begin{equation*}
\liminf\limits_{z\rightarrow 0}\frac{|f(z)|}{|z|}\,\le\,\tilde c_m\,\sigma_0^\frac{1}{m}\,,
\end{equation*}
which completes the proof. Here $\tilde c_m=c_{m+2}$ (from Theorem~\ref{th1}) depends only on $m.$
\end{proof}

\begin{example}
Assume that $m>0$ and $\kappa>0.$ Consider the equation
\begin{equation}
\label{exb1} f_{r}=\frac{-i}{\kappa \, r^{m+1}}|f_{\theta}|^{m} f_{\theta},
\end{equation}
in the unit disc.
\end{example}

Now $\sigma=-i \kappa^{-1} r^{-m-1}$ and satisfies the condition (\ref{condnb1}) with $\sigma_0=\kappa.$
It is easy to check that $f=\kappa^{\frac{1}{m}} r e^{i\theta}$ is a solution of the equation (\ref{exb1}), for which
\begin{equation*}
\lim\limits_{z\to 0}\frac{|f(z)|}{|z|}=\kappa^{\frac{1}{m}}<\infty.
\end{equation*}

\bigskip
\bigskip
{\textbf{Acknowledgment.}}\ \
The first author was supported by EU FP7 IRSES program STREV\-COMS, grant no. PIRSES-2013-612669.

\medskip
{\small \leftline{\textbf{Anatoly Golberg}} \em{
\leftline{Department of Mathematics,} \leftline{Holon
Institute of Technology,} \leftline{52 Golomb St., P.O.B. 305,}
\leftline{Holon 5810201, ISRAEL} \leftline{Fax: +972-3-5026615}
\leftline{e-mail: golberga@hit.ac.il}}}

\medskip

{\small \leftline{\textbf{Ruslan Salimov}}\em{ \leftline{Institute of Mathematics,}
\leftline {National Academy of Sciences of Ukraine,}
\leftline{3 Tereschenkivska St.,}
\leftline{Kiev-4 01601, UKRAINE}\leftline{e-mail:
ruslan.salimov1@gmail.com}}}

\medskip

{\small \leftline{\textbf{Maria Stefanchuk}}\em{ \leftline{Institute of Mathematics,}
\leftline {National Academy of Sciences of Ukraine,}
\leftline{3 Tereschenkivska St.,}
\leftline{Kiev-4 01601, UKRAINE}\leftline{e-mail:
stefanmv43@gmail.com}}}

\end{document}